\documentclass[article,12pt]{amsart}
\usepackage{amscd}
\usepackage{bbm}
\usepackage{dsfont}
\usepackage{enumerate}
\usepackage{latexsym}
\usepackage{srcltx}
\usepackage{amsfonts}
\usepackage{amssymb}
\usepackage{datetime}
\usepackage{setspace}
\usepackage{enumerate}

\newtheorem{theorem}{Theorem}
\newtheorem{proposition}[theorem]{Proposition}
\newtheorem{lemma}[theorem]{Lemma}
\newtheorem{corollary}[theorem]{Corollary}

\theoremstyle{definition}
\newtheorem{example}[theorem]{Example}

\newtheorem{problem} [theorem] {Problem}
\begin{document}

\baselineskip 14.85 pt

\centerline{ \textbf{Double commutants of multiplication operators on $C(K).$}}
\medskip

\centerline  {\bf  A. K. Kitover}

\centerline  {Department of Mathematics, CCP, Philadelphia, PA 19130, USA}
\bigskip

\textbf{Abstract.} Let $C(K)$ be the space of all real or complex valued continuous functions on a compact Hausdorff space $K$. We are interested in the following property of $K$: for any real valued $f \in C(K)$ the double commutant of the corresponding multiplication operator $F$ coincides with the norm closed algebra generated by $F$ and $I$. In this case we say that $K \in \mathcal{DCP}$. It was proved in~\cite{Ki} that any locally connected metrizable continuum is in $\mathcal{DCP}$. In this paper we indicate a class of arc connected but not locally connected continua that are in $\mathcal{DCP}$.
We also construct an example of a continuum that is not arc connected but is in $\mathcal{DCP}$.

\bigskip

\section{Introduction}

The famous von Neumann's double commutant theorem~\cite{N} can be stated the following way. Let $(X, \Sigma, \mu)$ be a space with measure and $f$ be a real-valued element of $L^\infty(X, \Sigma, \mu)$. Let $F$ be the corresponding multiplication operator in $L^2(X, \Sigma, \mu)$, i.e.
$(Fx)(t) = f(t)x(t)$ for $x \in L^2(X, \Sigma, \mu)$ and $t$ from a subset of full measure in $X$. Then
$$\{F\}^{cc} = \mathcal{A}_F$$
where $\{F\}^{cc}$ is the double commutant (or bicommutant)  of $F$, i.e. $\{F\}^{cc}$ consists of all bounded linear operators on $L^2(X, \Sigma, \mu)$ that commute with every operator commuting with $F$ and $\mathcal{A}_F$ is the closure in the weak (or strong) operator topology of the algebra generated by $F$ and the identity operator $I$.

The generalization on the case of complex multiplication operators (or normal operators on a Hilbert space) is then immediate.
Quite naturally arises the question of obtaining similar results for multiplication operators on other Banach spaces of functions. De Pagter and Ricker proved in~\cite{PaR} that von Neumann's result remains true for spaces $L^p(0,1), 1 \leq p < \infty$, and more generally for any Banach ideal $X$ in the space of all measurable functions such that $X$ has order continuous norm and $L^\infty(0,1) \subset X \subseteq L^1(0,1)$. But they also proved that the double commutant of the operator $T$, $(Tx)(t) = tx(t), x \in L^\infty, t \in [0,1]$, is considerably larger than the algebra $\mathcal{A}_T$ and consists of all operators of multiplication by Riemann integrable functions on $[0,1]$. The last result gives rise to the following question: let $C(K)$ be the space of all continuous real-valued functions on a Hausdorff compact space $K$. When is it true that for every multiplication operator $F$ on $C(K)$ its double commutant coincides with the algebra $\mathcal{A}_F$? This property is obviously a topological invariant of $K$ and we will denote  the class of compact Hausdorff spaces that have it as $\mathcal{DCP}$ (short for double commutant property).

\section{Continuums with $\mathcal{DCP}$ property}

 In~\cite{Ki} the author proved that if $K$ is a compact metrizable space without isolated points then the following implications hold.
 \begin{enumerate}
   \item If $K$ is connected and locally connected then $ K \in \mathcal{DCP}$.
   \item If $K \in \mathcal{DCP}$ then $K$ is connected.
 \end{enumerate}

 In the presence of isolated points the analogues of the above statements become more complicated (see~\cite[Theorem 1.15]{Ki}). To avoid these minor complications and keep closer to the essence of the problem we will assume that the compact spaces we consider have no isolated points.

 A simple example (see~\cite[Example 1.16]{Ki}) shows that the condition that $K$ is connected is not sufficient for $K \in \mathcal{DCP}$.
  \begin{example} \label{e1} Let $K$ be the closure in $\mathds{R}^2$ of the set $\{(x, \sin{1/x}) : x \in (0,1]\}$. Let $f(x,y) =x, (x,y) \in K$, and let $F$ be the corresponding multiplication operator. Then it is easy to see (see details in~\cite[Example 1.16]{Ki})  that the double commutant $\{F\}^{cc}$ consists of all operators of multiplication on functions from $C(K)$ but $\mathcal{A}_F$ consists of operators of multiplication on functions from $C(K)$ that are constant on the set $\{(0,y): y \in [0,1]\}$.
\end{example}
 Therefore the next question is whether the condition that $K$ is connected and locally connected is necessary for $K \in \mathcal{DCP}$? Below we provide a negative answer to this question. In order to consider the corresponding example let us recall the following two simple facts.

 \begin{proposition} \label{p1} Let $K$ be a compact Hausdorff space and $f \in C(K)$. Let $F$ be the corresponding multiplication operator. Then
 \begin{enumerate}
   \item The double commutant $\{F\}^{cc}$ consists of multiplication operators.
   \item The algebra $\mathcal{A}_F$ coincides with the closure of the algebra generated by $F$ and $I$ in the operator norm.
 \end{enumerate}

 \end{proposition}

 \begin{proof} $(1)$. Let $T \in \{F\}^{cc}$. Let $\mathbf{1}$ be the function in $C(K)$ identically equal to 1. Clearly for every $a \in C(K)$ the operator $F$ commutes with the multiplication operator $A$ where $Ax=ax, x \in C(K)$. Therefore for any $a \in C(K)$ $T$ commutes with $A$ and $T(a) = T(a\mathbf{1})= TA\mathbf{1} = AT\mathbf{1} = aT\mathbf{1} = (T\mathbf{1})a$. Hence if we take $g = T\mathbf{1}$ then $T$ coincides with
the multiplication operator $G$ generated by the function $g$.

\noindent $(2)$ If $T \in \{F\}^{cc}$ then by part $(1)$ of the proof $T=G$ where $G$ is a multiplication operator by a function $g \in C(K)$.       It remains to notice that $\|G\| = \|G \mathbf{1}\|_{C(K)}$ and therefore on $\{F\}^{cc}$ the convergence in strong operator topology implies convergence in the operator norm.
 \end{proof}

 \begin{corollary} \label{c1} Let $f \in C(K)$ and $F$ be the corresponding multiplication operator. The following two statements are equivalent.

 $(1)$ $\{F\}^{cc} = \mathcal{A}_F$.

 $(2)$ For any $G \in \{F\}^{cc}$ and for any $s, t \in K$ the implication holds
 $$f(s) = f(t) \Rightarrow g(s) = g(t),$$
 where $g \in C(K)$ is the function corresponding to the operator $G$.

 \end{corollary}

 In what follows our main tool will be the following lemma which was actually proved though not stated explicitly in~\cite{Ki} (see~\cite[Proof of Theorem 1.14]{Ki}).

 \begin{lemma} \label{l1} Let $K$ be a compact metrizable space, $f \in C(K)$, and $F$ be the corresponding multiplication operator. Let $G \in \{F\}^{cc}$ and $g$ be the corresponding function from $C(K)$. Let $u, v \in K$ be such that
 \begin{itemize}
  \item $f(u) = f(v)$.
   \item The points $u$ and $v$ have open, and locally connected neighborhoods in $K$.
   \item For any open connected neighborhood $U$ of $u$ there is an open interval $I_U$  in $\mathds{R}$  such that $f(u) \in I_U \subset f(U)$.
 \end{itemize}

 Then $g(u) = g(v)$.

 \end{lemma}

 We will also need a simple lemma proved in~\cite[Lemma 1.13]{Ki}

 \begin{lemma} \label{l2} Let $K$ be a compact Hausdorff space, $F, G$ multiplication operators on $C(K)$ by functions $f$ and $g$, respectively and $G \in \{F\}^{cc}$. Let $k \in K$ be such that $Int f^{-1}(\{f(k)\}) \neq \emptyset$.
 Then $g$ is constant on $f^{-1}(\{f(k)\})$.
 \end{lemma}

 Now we are ready to give an example of a metrizable connected compact space $K$ such that $K$ is not locally connected but $K \in \mathcal{DCP}$. Let $B$ be the well known \textbf{``broom''}.
 $$B = cl \{(x,y) \in \mathds{R}^2 : \; x \geq 0, \; y = \frac{1}{n} x, \; n \in \mathds{N}, \; x^2 + y^2 \leq 1 \}.$$

 \begin{proposition} \label{p2} $B \in \mathcal{DCP}$.

 \end{proposition}

 \begin{proof} Let $f \in C(B)$ and $G \in \{F\}^{cc}$. By part $(1)$ of Proposition~\ref{p1} $G$ is a multiplication operator. Let $g$ be the corresponding function from $C(K)$. Let $u, v \in B$ and $f(u) = f(v)$. We can assume without loss of generality that $f \geq 0$ and $\min \limits_{k \in B} f(k) = 0$. Let $D = \{k \in B: \; k = (x,0), 0 < x \leq 1\}$. We will divide the remaining part of the proof into four steps.

 $(I)$. Assume first that $u, v \in B \setminus D$ and that
 $0 < f(u) = f(v) < M = \max \limits_{k \in B} f(k)$. \footnote{ We can assume of course that $M > 0$ because otherwise $F = 0$ and the statement
 $\{F\}^{cc} = \mathcal{A}_F$ becomes trivial.} For any $m \in \mathds{N}$ let
 $B_m = \{(x,y) \in \mathds{R}^2 : \; x \geq 0, \; y = \frac{1}{n} x, \;  n \geq m, \; x^2 + y^2 \leq 1 \}$. Then for any large enough $m$ we have
 $$ \min \limits_{k \in B_m} f(k) < f(u) < \max \limits_{k \in B_m} f(k). \eqno{(1)}$$.
  Notice that for every $m \in \mathds{N}$ the set $B_m$ is a compact, connected and locally connected subset of $B$. Moreover, every point of $B_m$ is a point of local connectedness in $B$ and the set $B_m \setminus \{0,0\}$ is open in $B$. Let $B_m^1 = cl\{k \in B_m : \; f(k) < f(u)\}$ and $B_m^2 = cl\{k \in B_m : \; f(k) > f(u)\}$. There are two possibilities. $(a)$. The set $B^1_m \cap B^2_m$ is empty. In this case, because $B_m$ is connected, $f \equiv f(u)$ on some open subset of $B$ and by Lemma~\ref{l2} we have $g(u) = g(v)$.

 $(b)$. $\exists w \in B^1_m \cap B^2_m$. Because $B$ is locally connected at $w$ the pairs of points $(u,w)$ as well as $(v,w)$ satisfy all the conditions of Lemma~\ref{l1} whence $g(u) = g(v)$.

 $(II)$ Let $u, v \in B \setminus D$ and $f(u) = f(v) = 0$. There are two possibilities. First: $f \equiv 0$ on some open neighborhood of either $u$ or $v$. Then $g(u) = g(v)$ by Lemma~\ref{l2}. Second: $f$ is not constant on any open neighborhood of either $u$ or $v$. In this case, because $B \setminus D$ is locally connected, we can find sequences $u_n \mathop \rightarrow \limits_{n \to \infty} u$ and $v_n \mathop \rightarrow \limits_{n \to \infty} v$ such that $u_n, v_n \in B \setminus D$ and $0 < f(u_n) = f(v_n) < M, \; n \in \mathds{N}$. Then by the previous step $g(u_n) = g(v_n)$ whence $g(u) = g(v)$. The case $u,v \in B \setminus D$ and
 $f(u) = f(v) = M$ can be considered similarly.

 $(III)$. Now we will assume that $u$ and $v$ are arbitrary distinct points of $B$ and that $0 < f(u) = f(v) < M$. Let again $m \in \mathds{N}$ be so large that inequalities $(1)$ hold. Like on step $(I)$ we have two alternatives $(a)$ and $(b)$. In case $(a)$ we apply again Lemma~\ref{l2}. In case $(b)$ we cannot apply Lemma~\ref{l1} directly because $B$ might be not locally connected at $u$ and/or at $v$. Therefore we fix $w \in B^1_m \cap B^2_m$ and consider two subcases. $(b1)$. $f$ is constant on some neighborhood of either $u$ or $v$. Then $f(u) = f(v)$ by Lemma~\ref{l2}. $(b2)$. $f$ is not constant on any open neighborhood of $u$ or $v$. Let $u_n \in B \setminus D$ be such that
 $u_n \mathop \rightarrow \limits_{n \to \infty} u$. Because $f(w)$ is an inner in $\mathds{R}$ point of the set $f(W)$ where $W$ is an arbitrary connected open neighborhood of $w$ in $B_m$ we can find a sequence of points $w_n$ such that $w_n \in B_m \subset B \setminus D$ and for any large enough $n \in \mathds{N}$ we have $f(u_n) = f(w_n)$. Then by step $(I)$ $g(u_n) = g(w_n)$ whence $g(u) = g(w)$. Similarly we prove that $g(v) = g(w)$.

 $(IV)$. Finally assume that $u$ and $v$ are arbitrary points in $B$ and $f(u) = f(v) = 0$ (the case $f(u) = f(v) = M$ can be considered in the same way). If there is a point $w \in B \setminus D$ such that $f(w) =0$ then we can proceed as in step $(III)$. Let us assume therefore that $f > 0$ on $B \setminus D$. Let $a \in (0,1)$ be the smallest number such that $f(a,0) = 0$. Then for any $n \in \mathds{N}$ such that $n > 1/a$ the set $f(\{(x,0) : \; a - 1/n \leq x \leq a \})$ is an interval $[0, \delta_n]$ where $\delta_n > 0$. Therefore we can find
 $a_n \in [a - 1/n, a)$ and $u_n \in B \setminus D$ such that $u_n \mathop \rightarrow \limits_{n \to \infty} u$ and $f(u_n) = f(a_n,0), n \in \mathds{N}, \; n > 1/a$, whence by step $(III)$ $g((a_n,0) = g(u_n)$ and therefore $g((a,0) = g(u)$. Similarly, $g(v) = g(a,0) = g(u)$ and we are done.

 \end{proof}

 By analyzing the steps of the proof of Proposition~\ref{p2} and the properties of the broom $B$ we used, we come to the following more general statement that can be proved in exactly the same way as Proposition~\ref{p2}.

 \begin{proposition} \label{p5} Let $K$ be a compact connected metrizable space. Assume that there are compact subsets $K_m, m \in \mathds{N}$ of $K$ with the properties.
 \begin{enumerate}
   \item $K_m \varsubsetneqq K_{m+1}$.
   \item $K_m$ is connected and locally connected.
   \item The interior of $K_m$ in $K$ is dense in $K_m$, $m \in \mathds{N}$.
   \item Every point of $K_m$ is a point of local connectedness in $K$.
   \item The set $\bigcup \limits_{m=1}^\infty K_m $ is dense in $K$.
   \item For every point $k \in K \setminus \bigcup \limits_{m=1}^\infty K_m $ there is a path in $K$ from $k$ to a point in $\bigcup \limits_{m=1}^\infty K_m $.
 \end{enumerate}
 Then $K \in \mathcal{DCP}$.

 \end{proposition}

 \begin{example} \label{e2} This example is somewhat similar, though topologically not equivalent to the broom. The corresponding compact subspace of $\mathds{R}^2$ is traditionally called the \textbf{``bookcase''} and is defined as follows.
 $$ BC = cl \bigcup_{n=1}^\infty \{(x, 1/n) : \; x \in [0,1] \} \cup \{(0,y) : \; y \in [0,1]\}  \cup \{(1,y) : \; y \in [0,1]\} .$$
We claim that $BC \in \mathcal{DCP}$.
\end{example}

\begin{proof} For any $m \in \mathds{N}$ let $BC_m = BC \cap \{(x,y) \in \mathds{R}^2 : \; y \geq 1/m \}$. Then the compacts $BC_m$ have properties $(1) - (6)$ from the statement of Proposition~\ref{p5}.
\end{proof}

The conditions of Proposition~\ref{p5}  and the arc connectedness theorem (see~\cite[Theorem 5.1, page 36]{Why}) guarantee that the compact space $K$ satisfying the conditions of that proposition is arc connected. It is not known to the author if the arc connectedness of a metrizable compact $K$ is sufficient for $K \in \mathcal{DCP}$, but as our next example shows it surely is not necessary.

\begin{proposition} \label{p4} Let $K$ be the union of the square $[-1,0] \times [-1,1]$ and the set $\{(x, \sin{1/x}): \; 0 < x \leq 1 \}$.
Then $K \in \mathcal{DCP}$.

\end{proposition}

\begin{proof} Let $f \in C(K)$, $F$ be the corresponding multiplication operator, $G \in \{F\}^{cc}$, and $g \in C(K)$ the function corresponding to $G$. We can assume without loss of generality that $f(K) = [0, M]$ where $M > 0$. Let $E = \{(0,y) : \; y \in [-1,1]\}$. Notice that $K$ is locally connected at any point of $K \setminus E$. The set $K \setminus E$ is the union of two disjoint open connected subsets of $K$:
$C_1 = [-1,0) \times [-1,1]$ and $C_2 = \{(x, \sin{(1/x)}: x \in (0,1]\}$. Like in the proof of Proposition~\ref{p2} we have to consider several possibilities.
\begin{enumerate}
\item If $u, v \in C_1$ or $u, v \in C_2$ and $0 < f(u) = f(v) < M$. In this case we can prove that $g(u) = g(v)$ in very much the same  way as in step $(I)$ of the proof of Proposition~\ref{p2} by considering the sets $C_{1,m} = [-1 \times -1/m], m \in \mathds{N}$ (respectively the sets
$C_{2,m} = \{(x, \sin{(1/x)}) : \; 1/m < x < 1 \}$).

\item Let now assume that $0 < f(u) = f(v) < M$, $u \in C_1$, $v \in C_2$, and at least one of the inequalities holds $f(u) < \sup \limits_{k \in C_1}f(k)$ or $f(v) < \sup \limits_{k \in C_2}f(k)$. Then like in the proof of Proposition~\ref{p2} we can either find an open subset of $K$ on which $f$ is identically equal to $f(u)$ and apply Lemma~\ref{l2}, or there is a point $w \in C_1 \cup C_2 $ such that $f(w) = f(u)$ and for every open connected neighborhood $W$ of $w$ there is an open interval $I_w$ such that $f(w) \in I_w \subset f(W)$, and in this case we can apply Lemma~\ref{l1}.

\item Let $0 < f(u) = f(v) < M$, $u \in C_1$, $v \in C_2$, $f(u) = \sup \limits_{k \in C_1}f(k)$, and $f(v) = \inf \limits_{k \in C_2}f(k)$. It follows immediately that $f \equiv f(u) $ on $E$. For any $m \in \mathds{N}$ let $U_m$ be the open disk centered at $u$ and of radius $1/m$,
$V_m = \{(x, \sin{(1/x)}): x_v -1/m < x < x_v +1/m \}$ where $x_v$ is the $x$-coordinate of point $v$, and $W_m = K \cap (-1/m, 1/m) \times [-1,1]$. For any large enough $m$ the sets $U_m$, $V_m$, and $W_m$ are disjoint open neighborhoods in $K$ of $u$, $v$, and $E$, respectively. We can assume that $f$ is not identically equal to $f(u)$ on any open subset of $K$; indeed, otherwise we are done by Lemma~\ref{l2}. Then
$f(cl U_m) = [f(u) - \alpha_m, f(u)]$, $f(cl V_m) = [f(u), f(u) + \beta_m]$, and $f(cl W_m) = [f(u) - \gamma_m , f(u) + \delta_m]$, where
$\alpha_m, \beta_m, \gamma_m, \mathrm{and}\; \delta_m \searrow 0$. Therefore we can find points $u_m \in cl U_m$, $v_m \in cl V_m$, $w_m \in cl W_m \cap C_1$, and $z_m \in cl W_m \cap C_2$ such that $f(u_m) = f(w_m)$ and $f(v_m) = f(z_m)$. By part $(1)$ of the proof $g(u_m) = g(w_m)$ and $g(v_m) = g(z_m)$. Let $w$ (respectively, $z$) be a limit point of the sequence $w_m$ (respectively, $z_m$). Then $w, z \in E$, $g(u) =g(w)$, and
$g(v) = g(z)$. It remains to prove that $g(w) = g(z)$. If $w = z$ there is nothing to prove, therefore assume that $w \neq z$. Let $A_m$ (respectively, $B_m$) be the intersection of the closed disk with the center $w$ (respectively, z) and of radius $1/m$ with the closure of $C_2$. Recalling our assumption that $f$ is not identically equal to $f(u)$ on any open subset of $K$ we see that we can find sequences $a_m, b_m \in C_1$ that converge to $w$ (respectively, to $z$) and such that $f(a_m) = f(b_m)$. By step 1 $g(a_m) = g(b_m)$ whence $g(w) = g(z)$.

\item The implications $f(u) = f(v) =0 \Rightarrow g(u) = g(v)$ and $f(u) = f(v) =M \Rightarrow g(u) = g(v)$ can be easily proved by using the same type of reasoning as in parts $(1)$ - $(3)$.
\end{enumerate}

\end{proof}

Finally let us state some open questions.

\begin{problem} \label{pr1}

\begin{enumerate}
\item Is it possible to characterize the metrizable continua from the class $\mathcal{DCP}$ in purely topological terms not involving multiplication operators?
 \item In particular, is it true that any metrizable arc connected continuum belongs to $\mathcal{DCP}$?

 \item This question is a special case of the previous one. Let $C$ be the standard Cantor set and
$$K = \{(x,y): \; x \in C, y \in [0,1] \} \cup \{(x,0): x \in [0,1]\}. $$
Is it true that $K \in \mathcal{DCP}$?
  A positive answer to question $(3)$ would be in the author's opinion a strong indication that the answer to question $(2)$ should also be positive.

\end{enumerate}

\end{problem}

\end{document}